\documentclass[a4paper, 12pt]{amsart}
\usepackage{amsmath, amssymb, amscd, enumerate, hyperref, MnSymbol, wasysym, tikz}
\usepackage[all,cmtip]{xy}

\newtheorem{theorem}{Theorem}[section]
\newtheorem{proposition}[theorem]{Proposition}

\theoremstyle{remark}
\newtheorem{remark}{Remark}
\newtheorem{example}{Example}

\makeatletter
\def\@eqnnum{(\theequation)}
\makeatother

\numberwithin{equation}{section}

\begin{document}


\title[Representations by quadratic forms with congruence conditions]
{On the number of representations of integers by quadratic forms with congruence conditions}

\author{Bumkyu Cho}
\address{Department of Mathematics, Dongguk University, 30 Pildong-ro 1-gil, Jung-gu, Seoul, 100-715, Republic of Korea}

\email{bam@dongguk.edu}

\subjclass[2010]{Primary 11N32; Secondary 11R37, 11F11}

\thanks{The author was supported by the Dongguk University Research Fund of 2014.}

\keywords{}

\dedicatory{}

\begin{abstract}
We characterize the generating functions of the numbers of representations described in the title in the context of modular forms. Appealing to this characterization we obtain explicit formulas for the representation numbers as examples.
\end{abstract}

\maketitle

\section{Introduction and statements of results}

In the previous work \cite{Cho1, Cho2, Cho3}, the author exploited class field theory to obtain characterizations of integers that can be expressed as $x^2 + ny^2$ or $x^2 + xy + ny^2$ ($n \in \mathbb N$) with extra conditions $x \equiv 1 \pmod m$, $y \equiv 0 \pmod m$ on the variables. For example we have for $a \in \mathbb N$ with $(a, 6) = 1$, $a$ is of the form $x^2 + y^2$ with $x \equiv 1 \pmod 3$, $y \equiv 0 \pmod 3$ if and only if $a$ has exactly even number of prime divisors that are congruent to $5$ modulo $12$ (see \cite[Example 2]{Cho2}). It is then natural to study on the number of such representations, i.e. to try to find the formula for
\[ r(a) \, := \, |\{ (x, y) \in \mathbb Z^2 \, | \, a = x^2 + y^2, \ x \equiv 1 \ (\bmod \ 3), \ y \equiv 0 \ (\bmod \ 3) \}| \]
where $a$ is a positive integer (see Example 1 in Section 3).

Up to the author's knowledge there are only a few results in this direction. To introduce those we first quote Jacobi's four-square theorem. If we let
\[ r_4(n) = |\{ (x_1, x_2, x_3, x_4) \in \mathbb Z^4 \, | \, x_1^2 + x_2^2 + x_3^2 + x_4^2 = n \}| \]
then he showed in 1834 that $r_4(n) = 8 \sigma(n) - 32 \sigma(n/4)$ where $\sigma(n) := \sum_{d|n} d$. In \cite[Vorlesung 11]{Hur} Hurwitz could also prove Jacobi's theorem by making use of the arithmetic of quaternions. Furthermore he was able to obtain that the number of representations of $4n$ ($n$ odd) as the sum of four odd squares is given as $16 \sigma(n)$. Recently, Deutsch \cite{Deu} also gave a quaternionic proof of the formula for the number of representations of a positive integer by the quadratic form $x_1^2 + x_2^2 + 2x_3^2 + 2x_4^2$. His approach also produces formulas for the numbers of representations of $4n$ and $8n$ ($n$ odd) by the same quadratic form with certain restrictions on the parity of the variables $x_1$, $x_2$, $x_3$, and $x_4$. For more details, see \cite[Theorem 58]{Deu}.

The purpose of this article is to find such kind of formulas by characterizing the generating functions of the representation numbers in the context of modular forms. For example we will obtain explicit formulas for the representation numbers
\[ r_Q^{u_1, u_2, u_3, u_4; 3}(n) \, := \, |\{ (x_1, x_2, x_3, x_4) \in \mathbb Z^4 \, | \, x_1^2 + x_2^2 + x_3^2 + x_4^2 = n, \ x_i \equiv u_i \ (\bmod \ 3) \}| \]
where $u_i \in \{ 0, 1, 2 \}$ (see Example 5 in Section 3).

\medskip

We are now going to introduce our theorems. Let $\mathbb H = \{ \tau \in \mathbb C \, | \, \mbox{Im} (\tau) > 0 \}$ be the complex upper half plane and $\mathbb H^\ast = \mathbb H \cup \mathbb Q \cup \{ \infty \}$. Here the elements of $\mathbb Q \cup \{ \infty \}$ are called cusps. The group $SL_2(\mathbb Z)$ acts on $\mathbb H^\ast$ by linear
fractional transformation $\gamma(\tau) = \frac{a \tau + b} {c
\tau + d}$ for $\gamma = ( \substack{ a \ b \\ c \ d} ) \in SL_2
(\mathbb Z)$. The principal congruence
subgroup $\Gamma(N)$ with $N \in \mathbb N$ is defined to be
\[ \Gamma(N) = \{ \big(\begin{smallmatrix} a & b \\ c & d \end{smallmatrix}\big) \in
SL_2 (\mathbb Z) \, | \, a \equiv d \equiv 1 \ (\bmod \ N), \ b \equiv c
\equiv 0 \ (\bmod \ N) \} \] and any subgroup of $SL_2(\mathbb
Z)$ containing $\Gamma(N)$ is called a congruence
subgroup of level $N$. In this article we mainly employ the congruence subgroups
$\Gamma_0(N)$ and $\Gamma_1(N)$ that are defined as follows: $\Gamma_0(N)$
(respectively, $\Gamma_1(N)$) consists of all $( \substack{ a \ b \\
c \ d} ) \in SL_2(\mathbb Z)$ such that $c \equiv 0 \pmod N$
(respectively, $c \equiv 0 \pmod N$ and $a \equiv d \equiv 1 \pmod
N$).

Let $k, N$ be positive integers and $\psi$ a Dirichlet character modulo $N$. The slash operator $|_{k, \psi}$ is defined as
\[ (f|_{k, \psi}\gamma)(\tau) = \bar{\psi}(d) (c \tau + d)^{-k} f(\gamma\tau) \]
where $f$ is a meromorphic function on $\mathbb H$ and $\gamma = (\begin{smallmatrix} a & b \\ c & d \end{smallmatrix}) \in \Gamma_0(N)$. It is tedious to verify that $(f|_{k, \psi} \gamma)|_{k, \psi} \gamma' = f|_{k, \psi} (\gamma \gamma')$ for $\gamma, \gamma' \in \Gamma_0(N)$.

Let $k$, $N$, $\psi$ be as above and let $\Gamma$ be a congruence subgroup of level $N$ contained in $\Gamma_0(N)$. By definition a modular form of weight $k$ for $\Gamma$ with nebentypus $\psi$ is a function $f$ satisfying

(1) $f$ is holomorphic on $\mathbb H$

(2) $f |_{k, \psi} \gamma = f$ for all $\gamma \in \Gamma$

(3) $f$ is holomorphic at all cusps.

If a modular form $f$ vanishes at all cusps, then it is called a cusp form. The $\mathbb C$-vector space of modular forms of weight $k$ for
$\Gamma$ with nebentypus $\psi$ is denoted $M_k(\Gamma, \psi)$. The subspace consisting of Eisenstein series (respectively, cusp forms) is denoted by $E_k(\Gamma, \psi)$ (respectively, $S_k(\Gamma, \psi)$). In case $\psi$ is the trivial character, we simply denote them by $M_k(\Gamma)$, $E_k(\Gamma)$, and $S_k(\Gamma)$. We also employ the standard notation $q = e^{2 \pi i \tau}$ for $\tau \in \mathbb H$.

\medskip

Let $Q(x) = \sum_{1 \leq i \leq j \leq n} a_{ij} x_i x_j$ be a positive-definite quadratic form where $n \in \mathbb N$ is even, $a_{ij} \in \mathbb Z$, and $x = {^t}(x_1 \ldots x_n)$. We denote by $A_Q$ the symmetric matrix of size $n$ over $\mathbb Z$ whose entries satisfy the following conditions:
\[ (A_Q)_{ij} = a_{ij} \quad \mbox{for } i < j, \qquad (A_Q)_{ii} = 2a_{ii}. \]
Observe that $Q(x) = \frac{1}{2} {^t}x A_Q x$. Let $D_Q = (-1)^{n/2} \det(A_Q)$ be the discriminant of the quadratic form $Q$. The level of the quadratic form $Q$ is defined to be the smallest integer $N_Q \in \mathbb N$ such that $N_QA_Q^{-1} \in \mathrm{M}_n(\mathbb Z)$ and all the diagonal entries of $N_Q A_Q^{-1}$ are even. In fact, it is known that
\[ N_Q = \frac{|D_Q|}{\gcd(c_{ij}, \frac{c_{ii}}{2} \, | \, 1 \leq i, j \leq n)} \]
where $c_{ij}$ denotes the $(i, j)$ cofactor of $A_Q$ (see Section 2).

We are now ready to introduce our theorems.

\begin{theorem}\label{Theorem - main}
Let $Q(x) = \sum_{1 \leq i \leq j \leq n} a_{ij} x_i x_j$ $(n$ even$)$ be a positive-definite integral quadratic form and $P(x)$ a spherical function of degree $\nu$ with respect to the coefficient matrix $A_Q$. For every $m \in \mathbb N$ and every $u \in \mathbb Z^n$, the function
\[ f_{Q, P}^{u; m}(\tau) \ := \ \sum_{x \in \mathbb Z^n \atop x \equiv u \,\, (\mathrm{mod} \, m)} P(x) q^{Q(x)} \]
is a modular form of weight $\frac{n}{2} + \nu$ for $\Gamma_0(m^2 N_Q) \cap \Gamma_1(m)$ with nebentypus $(\frac{D_Q}{\cdot})$. Furthermore, if $\nu \geq 1$, then it is a cusp form. Here $(\frac{D_Q}{\cdot})$ denotes the Kronecker symbol.
\end{theorem}


\medskip

When $P(x) = 1$, we simply write $f_Q^{u; m}(\tau)$ instead of $f_{Q, 1}^{u; m}(\tau)$. If we put
\[ r_Q^{u; m}(k) \ := \ | \{ (x_1, \ldots, x_n) \in \mathbb Z^n \, | \, Q(x_1, \ldots, x_n) = k, \ x \equiv u \ (\bmod \ m) \} | \]
then $f_Q^{u; m}(\tau)$ is the generating function of the representation numbers $r_Q^{u; m}(k)$, i.e.
\[ f_Q^{u; m}(\tau) \ = \ \sum_{k = 0}^{\infty} r_Q^{u; m}(k) q^{k}. \]
The next theorem allows us to decompose the space $M_{n/2 + \nu}(\Gamma_0(m^2 N_Q) \cap \Gamma_1(m), (\frac{D_Q}{\cdot}))$ into $\chi$-eigenspaces for the congruence subgroup $\Gamma_0(m^2 N_Q)$.

\begin{theorem}\label{Theorem - decomposition}
Let $k, M, N$ be positive integers and $\psi$ a Dirichlet character modulo $MN$. Then we have
\[ M_{k}(\Gamma_0(MN) \cap \Gamma_1(M), \psi ) \ \cong \ \bigoplus_{\chi} M_{k}(\Gamma_0(MN), \chi\psi ) \]
where $\chi$ runs over all Dirichlet characters modulo $M$ such that $\chi(-1) = (-1)^k \psi (-1)$.
\end{theorem}

\section{Proofs of Theorems \ref{Theorem - main} and \ref{Theorem - decomposition}}

We begin by summarizing several facts about theta functions associated with quadratic forms. The reader may refer to \cite[Chapter IX]{Sch} or \cite[Section 4.9]{Miy}.

\medskip

Let $A \in \mathrm{M}_n(\mathbb Z)$ ($n$ even) be a positive-definite matrix whose entries satisfy the following conditions:
\[ A_{ij} = A_{ji}, \qquad A_{ii} \equiv 0 \ (\bmod \ 2). \]
Such matrices are called even. Then the quadratic form
\[ Q(x) := \frac{1}{2} {^t}xAx = \sum_{1 \leq i \leq j \leq n} b_{ij} x_ix_j , \qquad x = {^t}(x_1 \ldots x_n) \]
is positive-definite and
\[ b_{ij} = A_{ij} \quad \mbox{for } i < j, \qquad b_{ii} = \frac{1}{2} A_{ii}. \]

Let $D = (-1)^{n/2} \det(A)$ be the discriminant of the quadratic form $Q$. For example, $Q(x) = ax_1^2 + bx_1x_2 + cx_2^2$ has discriminant $b^2 - 4ac$. The level of the quadratic form $Q$ is defined to be the smallest integer $N \in \mathbb N$ such that $NA^{-1}$ is an even matrix.

\begin{proposition}\label{Proposition - level}
The level of the quadratic form $Q$ is given as
\[ N = \frac{|D|}{\gcd(c_{ij}, \frac{c_{ii}}{2} \, | \, 1 \leq i, j \leq n)} \]
where $c_{ij}$ is the $(i, j)$ cofactor of $A$.
\end{proposition}

\begin{proof}
See \cite[Theorem 1 in Chapter IX]{Sch}.
\end{proof}

We further remark that $N$ and $D$ have the same prime divisors. Let $P(x)$ be a spherical function of degree $\nu$ with respect to $A$. It is a homogeneous polynomial of degree $\nu$ in variables $x_1, \ldots, x_n$ (see \cite[p.186]{Miy}) given as
\[
P(x) = \left\{ \begin{array}{ll}
\mbox{a constant} & \mbox{if $\nu = 0$}, \\
{^t}l A x \quad (l \in \mathbb C^n) & \mbox{if $\nu = 1$}, \\
\mbox{a linear combination of } ({^t}lAx)^\nu \quad (l \in \mathbb C^n, \ {^t}lAl = 0) & \mbox{if $\nu > 1$}.
\end{array} \right.
\]
Let $h \in \mathbb Z^n$ be a column vector satisfying $Ah \equiv 0 \pmod N$. We now define the theta function $\theta(\tau; h, A, N, P)$ by
\[ \theta(\tau; h, A, N, P) = \sum_{x \in \mathbb Z^n \atop x \equiv h \,\, (N)} P(x) e(A[x]\tau/2N^2) \]
where $A[x] = {^t x}Ax$ and $e(z) = e^{2\pi i z}$. Then the theta functions are holomorphic in $\mathbb H$ and 
\begin{eqnarray*}
& \theta(\tau; h_1, A, N, P) = \theta(\tau; h_2, A, N, P) & \mbox{if } h_1 \equiv h_2 \ (\bmod \ {N}), \\
& \theta(\tau; -h, A, N, P) = (-1)^\nu \theta(\tau; h, A, N, P). &
\end{eqnarray*}

\begin{proposition}\label{Proposition - cusp}
We have the transformation formulas
\begin{eqnarray*}
\theta(\tau + 1; h, A, N, P) & = & e(A[h]/2N^2) \theta(\tau; h, A, N, P), \\
\theta(-1/\tau; h, A, N, P) & = & (-i)^{n/2 + 2\nu} |D|^{-1/2} \tau^{n/2 + \nu} \sum_{l \in \mathbb Z^n / N\mathbb Z^n \atop A l \equiv 0 \,\, (N)} e({^t l}Ah/N^2) \theta(\tau; l, A, N, P).
\end{eqnarray*}
\end{proposition}

\begin{proof}
We refer to \cite[Theorem 2 in Chapter IX]{Sch}.
\end{proof}

The theta functions behave in a relatively simple way under certain modular transformations.

\begin{proposition}\label{Proposition - transformation}
For any $\gamma = (\begin{smallmatrix} a & b \\ c & d \end{smallmatrix}) \in \Gamma_0(N)$ we have
\[
\theta(\gamma\tau; h, A, N, P) = \big( \frac{D}{d} \big) e(ab A[h]/2N^2) (c\tau + d)^{n/2 + \nu} \theta(\tau; ah, A, N, P).
\]
\end{proposition}

\begin{proof}
See \cite[Theorem 5 in Chapter IX]{Sch}.
\end{proof}

We are now ready to prove our theorems.

\begin{proof}[Proof of Theorem \ref{Theorem - main}]
Appealing to Proposition \ref{Proposition - level} we easily see that the level of $mA_Q$ is $m N_Q$, whence we may take $A = mA_Q$, $N = m N_Q$, $h = uN_Q$ in Proposition \ref{Proposition - transformation}. Since
\begin{eqnarray*}
\theta(m\tau; uN_Q, mA_Q, mN_Q, P) & = & \sum_{x \in \mathbb Z^n \atop x \equiv uN_Q \,\, (mN_Q)} P(x) e(A_Q[x]\tau/2N_Q^2) \\
& = & N_Q^\nu \sum_{x \in \mathbb Z^n \atop x \equiv u \,\, (m)} P(x) q^{Q(x)},
\end{eqnarray*}
we need to show that
\[ \theta(m\tau; uN_Q, mA_Q, mN_Q, P) \ \in \ M_{n/2 + \nu}(\Gamma_0(m^2 N_Q) \cap \Gamma_1(m), (\frac{D_Q}{\cdot})). \]
For every $\gamma = (\begin{smallmatrix} a & b \\ c & d \end{smallmatrix}) \in \Gamma_0(m^2 N_Q) \cap \Gamma_1(m)$, one can deduce by Proposition \ref{Proposition - transformation} that
\begin{eqnarray*}
& & \theta(m(\gamma \tau); uN_Q, mA_Q, mN_Q, P) \\
& = & \big( \frac{m^n D_Q}{d} \big) e(ab A_Q[u]/2) (c\tau + d)^{n/2 + \nu} \theta(m\tau; auN_Q, mA_Q, mN_Q, P) \\
& = & \big( \frac{D_Q}{d} \big) (c\tau + d)^{n/2 + \nu} \theta(m\tau; uN_Q, mA_Q, mN_Q, P)
\end{eqnarray*}
because $auN_Q \equiv uN_Q \pmod {mN_Q}$. Proposition \ref{Proposition - cusp} implies that the theta functions are holomorphic (respectively, have zeros if $\nu \geq 1$) at all cusps, whence the proof is complete.
\end{proof}

\begin{proof}[Proof of Theorem \ref{Theorem - decomposition}]
For $f \in M_{k}(\Gamma_0(MN) \cap \Gamma_1(M), \psi )$ and $\chi$ a Dirichlet character modulo $M$, we define
\[ f_\chi (\tau) = \frac{1}{\varphi(M)} \sum_{d \in (\mathbb Z / M\mathbb Z)^\times} \bar{\chi}(d) (f|_{k, \psi}\gamma_d)(\tau) \]
where $\gamma_d = (\begin{smallmatrix} * & * \\ * & d_0 \end{smallmatrix}) \in \Gamma_0(MN)$ is any fixed matrix with $d_0 \equiv d \pmod M$. We then have for any $\gamma_{d'} \in \Gamma_0(MN)$,
\begin{eqnarray*}
(f_\chi|_{k, \psi}\gamma_{d'})(\tau) & = & \frac{1}{\varphi(M)} \sum_{d \in (\mathbb Z / M\mathbb Z)^\times} \bar{\chi}(d) (f|_{k, \psi}\gamma_{dd'})(\tau) \\
& = & \chi(d') f_\chi(\tau),
\end{eqnarray*}
whence $f_\chi$ is contained in $M_k(\Gamma_0(MN), \chi \psi)$. One has
\begin{eqnarray*}
\sum_{\chi} f_\chi(\tau) & = & \frac{1}{\varphi(M)} \sum_{d \in (\mathbb Z / M\mathbb Z)^\times} (f|_{k, \psi}\gamma_d)(\tau) \sum_{\chi} \bar{\chi}(d) \\
& = & f(\tau)
\end{eqnarray*}
where the summation is taken over all Dirichlet characters $\chi$ modulo $M$. This is because
\[ \sum_{\chi} \bar{\chi}(d) = \left\{
\begin{array}{ll}
\varphi(M) & \mbox{if } d = 1, \\
0 & \mbox{otherwise.}
\end{array}
\right. \]
Observe that $M_k(\Gamma_0(MN), \chi_1 \psi) \cap M_k(\Gamma_0(MN), \chi_2 \psi) = \{ 0 \}$ if $\chi_1 \neq \chi_2$. We thus have the decomposition
\[ M_{k}(\Gamma_0(MN) \cap \Gamma_1(M), \psi ) \ \cong \ \bigoplus_{\chi} M_{k}(\Gamma_0(MN), \chi\psi ) \]
where $\chi$ runs over all Dirichlet characters modulo $M$. By definition one has
\[ M_k(\Gamma_0(MN), \chi \psi) = \{0\} \]
unless $(\chi \psi)(-1) = (-1)^k$. This completes the proof.
\end{proof}

\section{Examples}

In the case when the space $S_{n/2}(\Gamma_0(m^2 N_Q) \cap \Gamma_1(m), (\frac{D_Q}{\cdot}))$ of cusp forms is trivial, we can get an explicit formula for the number of representations of integers by the quadratic form $Q(x)$ with congruence condition $x \equiv u \pmod m$ on the variables. This is because the space $M_{n/2}(\Gamma_0(m^2 N_Q) \cap \Gamma_1(m), (\frac{D_Q}{\cdot}))$ is then spanned by Eisenstein series and their Fourier coefficients are explicitly known (see Theorem \ref{Theorem - decomposition} and the two theorems below).

\medskip

Let $A_{N, 1}$ be the set of triples $(\{\psi, \varphi \}, t)$ such that $\psi$ and $\varphi$, taken as an unordered pair, are primitive Dirichlet characters satisfying $(\psi \varphi)(-1) = -1$, and $t$ is a positive integer such that $tuv | N$, where $u$ and $v$ denote the conductors of $\psi$ and $\varphi$, respectively. For $(\{\psi, \varphi \}, t) \in A_{N, 1}$ let
\[ E_{1}^{\psi, \varphi, t}(\tau) \, = \, \frac{1}{2}\big( \delta(\varphi) L(0, \psi) + \delta(\psi) L(0, \varphi) \big) + \sum_{n = 1}^\infty \sigma_0^{\psi, \varphi} (n) q^{tn}, \]
where $\delta(\psi)$ is $1$ if $\psi$ is the trivial character and is $0$ otherwise, and
\[ \sigma_{0}^{\psi, \varphi}(n) \, = \, \sum_{d|n} \psi(n/d) \varphi(d). \]

\begin{theorem}\cite[Theorem 4.8.1]{DS}
For any Dirichlet character $\chi$ modulo $N$, the set
\[ \{ E_{1}^{\psi, \varphi, t} \, | \, (\{ \psi, \varphi \}, t) \in A_{N, 1}, \ \psi \varphi = \chi \} \]
represents a basis for the space $E_1(\Gamma_0(N), \chi)$ of Eisenstein series of weight one with nebentypus $\chi$.
\end{theorem}

Hereafter $\chi_n(\cdot)$ denotes the quadratic character $\big( \frac{n}{\cdot} \big)$ for $n \in \mathbb Z$.

\begin{example}
For $Q(x) = x_1^2 + x_2^2$ and $u_i \in \{0, 1, 2 \}$, let
\[ f_Q^{u_1, u_2; 3}(\tau) = \sum_{x_i \equiv u_i \,\, (3)} q^{x_1^2 + x_2^2} = \sum_{n = 0}^{\infty} r_Q^{u_1, u_2; 3}(n) q^n. \]
Appealing to Theorems \ref{Theorem - main} and \ref{Theorem - decomposition} we see that $f_Q^{u_1, u_2; 3}(\tau)$ is contained in $M_{1}(\Gamma_0(36), \chi_{-4})$. This space has no cusp forms other than $0$, whence we see by the preceding theorem that $M_{1}(\Gamma_0(36), \chi_{-4})$ is decomposed as
\[ M_{1}(\Gamma_0(36), \chi_{-4}) = \mathbb C E_1^{1, \chi_{-4}, 1} \oplus \mathbb C E_1^{1, \chi_{-4}, 3} \oplus \mathbb C E_1^{1, \chi_{-4}, 9} \oplus \mathbb C E_1^{\chi_{12}, \chi_{-3}, 1}. \]
Their Fourier expansions are explicitly given as
\begin{eqnarray*}
E_1^{1, \chi_{-4}, 1}(\tau) & = & \frac{1}{4} + \sum_{n=1}^\infty \big( \sum_{d|n} \big( \frac{-4}{d} \big) \big) q^n, \\
E_1^{\chi_{12}, \chi_{-3}, 1}(\tau) & = & \sum_{n=1}^\infty \big( \sum_{d|n} \big(\frac{12}{n/d}\big)\big(\frac{-3}{d}\big) \big) q^n.
\end{eqnarray*}
Calculating sufficiently many $r_Q^{u_1, u_2; 3}(n)$, we can express $f_Q^{u_1, u_2; 3}(\tau)$ as a linear combination of the above Eisenstein series. In terms of the identity
\[ \sum_{d|n} \big(\frac{12}{n/d}\big)\big(\frac{-3}{d}\big) \ = \ \big( \frac{3}{n} \big) \sum_{d|n \atop \frac{n}{d} \, \mathrm{odd}} \big( \frac{-1}{d} \big) \]
we obtain for every $n \in \mathbb N$
\begin{eqnarray*}
r_Q^{1, 0; 3}(n) & = & \frac{1}{2} \sum_{d|n} \big( \frac{-4}{d} \big) - \frac{1}{2} \sum_{d|\frac{n}{9}} \big( \frac{-4}{d} \big) + \frac{1}{2} \big( \frac{3}{n} \big) \sum_{d|n \atop \frac{n}{d} \, \mathrm{odd}} \big( \frac{-1}{d} \big), \\
r_Q^{1, 1; 3}(n) & = & \frac{1}{2} \sum_{d|n} \big( \frac{-4}{d} \big) - \frac{1}{2} \sum_{d|\frac{n}{9}} \big( \frac{-4}{d} \big) - \frac{1}{2} \big( \frac{3}{n} \big) \sum_{d|n \atop \frac{n}{d} \, \mathrm{odd}} \big( \frac{-1}{d} \big).
\end{eqnarray*}
\end{example}

\begin{example}
The following results are easily obtained by the same method as above, whence we merely present their formulas without detailed explanation.

(1) If $Q(x) = x_1^2 + x_2^2$, then one has for any $n \in \mathbb N$
\begin{eqnarray*}
r_Q^{1, 0; 2}(n) & = & 2 \sum_{d|n} \big( \frac{-4}{d} \big) - 2 \sum_{d|\frac{n}{2}} \big( \frac{-4}{d} \big), \\
r_Q^{1, 1; 2}(n) & = & 4 \sum_{d|\frac{n}{2}} \big( \frac{-4}{d} \big) - 4 \sum_{d|\frac{n}{4}} \big( \frac{-4}{d} \big).
\end{eqnarray*}

(2) For $Q(x) = x_1^2 + x_1x_2 + x_2^2$, we see $r_Q^{u_1, u_2; 2}(n) = r_Q^{u_2, u_1; 2}(n) = r_Q^{u_1+u_2, -u_2; 2}(n)$ and $r_Q^{u_1, u_2; 3}(n) = r_Q^{u_2, u_1; 3}(n) = r_Q^{-u_1, -u_2; 3}(n) = r_Q^{u_1+u_2, -u_2; 3}(n)$. One has for every $n \in \mathbb N$
\begin{eqnarray*}
r_Q^{1, 0; 2}(n) & = & 2 \sum_{d|n} \big( \frac{-3}{d} \big) - 2 \sum_{d|\frac{n}{4}} \big( \frac{-3}{d} \big), \\
r_Q^{1, 0; 3}(n) & = & \sum_{d|n} \big( \frac{-3}{d} \big) - \sum_{d|\frac{n}{3}} \big( \frac{-3}{d} \big), \\
r_Q^{1, 1; 3}(n) & = & 3 \sum_{d|\frac{n}{3}} \big( \frac{-3}{d} \big) - 3 \sum_{d|\frac{n}{9}} \big( \frac{-3}{d} \big).
\end{eqnarray*}

(3) Given $Q(x) = x_1^2 + 2x_2^2$, one has for any $n \in \mathbb N$
\begin{eqnarray*}
r_Q^{1, 0; 2}(n) & = & \sum_{d|n} \big(\frac{-8}{d}\big) - \sum_{d|\frac{n}{2}} \big(\frac{-8}{d}\big) + \sum_{d|n} \big(\frac{-4}{n/d}\big) \big(\frac{8}{d}\big), \\
r_Q^{0, 1; 2}(n) & = & 2\sum_{d|\frac{n}{2}} \big(\frac{-8}{d}\big) - 2\sum_{d|\frac{n}{4}} \big(\frac{-8}{d}\big), \\
r_Q^{1, 1; 2}(n) & = & \sum_{d|n} \big(\frac{-8}{d}\big) - \sum_{d|\frac{n}{2}} \big(\frac{-8}{d}\big) - \sum_{d|n} \big(\frac{-4}{n/d}\big) \big(\frac{8}{d}\big).
\end{eqnarray*}

(4) When $Q(x) = x_1^2 + x_1x_2 + 2x_2^2$ we see $r_Q^{u_1, u_2; 2}(n) = r_Q^{u_1+u_2, -u_2; 2}(n)$ and obtain for any $n \in \mathbb N$
\begin{eqnarray*}
r_Q^{1, 0; 2}(n) & = & 2 \sum_{d|n} \big( \frac{-7}{d} \big) - 4 \sum_{d|\frac{n}{2}} \big( \frac{-7}{d} \big) + 2 \sum_{d|\frac{n}{4}} \big(\frac{-7}{d}\big), \\
r_Q^{0, 1; 2}(n) & = & 2 \sum_{d|\frac{n}{2}} \big( \frac{-7}{d} \big) - 2 \sum_{d|\frac{n}{4}} \big(\frac{-7}{d}\big).
\end{eqnarray*}
\end{example}

\medskip

We are now going to deal with quadratic forms of four variables. Let $A_{N, 2}$ be the set of triples $(\psi, \varphi, t)$ such that $\psi$ and $\varphi$, taken this time as an ordered pair, are primitive Dirichlet characters satisfying $(\psi \varphi)(-1) = 1$, and $t$ is an integer such that $1 < tuv | N$, where $u$ and $v$ denote the conductors of $\psi$ and $\varphi$, respectively. For any triple $(\psi, \varphi, t) \in A_{N, 2}$ we define
\[ E_{2}^{\psi, \varphi, t}(\tau) \, = \, \left\{ \begin{array}{ll}
\sum_{n=1}^{\infty} \sigma(n) q^n - t \sum_{n=1}^{\infty} \sigma(n) q^{tn} & \mbox{if } \psi = \varphi = 1 \\
\frac{1}{2}\delta(\psi) L(-1, \varphi) + \sum_{n = 1}^\infty \sigma_1^{\psi, \varphi} (n) q^{tn} & \mbox{otherwise}.
\end{array} \right. \]
Here $\delta(\psi)$ has the same meaning as before, and
\[ \sigma_{1}^{\psi, \varphi}(n) \, = \, \sum_{d|n} \psi(n/d) \varphi(d) d. \]


\begin{theorem}\cite[Theorem 4.6.2]{DS}
For any Dirichlet character $\chi$ modulo $N$, the set
\[ \{ E_{2}^{\psi, \varphi, t} \, | \, ( \psi, \varphi, t) \in A_{N, 2}, \ \psi \varphi = \chi \} \]
represents a basis for the space $E_2(\Gamma_0(N), \chi)$ of Eisenstein series of weight two with nebentypus $\chi$.
\end{theorem}

\begin{remark}
We will not deal with an example of weight greater than $2$, whence we merely refer the reader to \cite[Theorem 4.5.2]{DS} for the basis of $E_k(\Gamma_0(N), \chi)$ with $k \geq 3$.
\end{remark}

\begin{example}
If we let $Q(x) = x_1^2 + x_2^2 + x_3^2 + x_4^2$, then $f_Q^{u; 2}(\tau)$ is a modular form in $M_2(\Gamma_0(16))$ by Theorem \ref{Theorem - main}. Appealing to the theory of modular forms we see $S_2(\Gamma_0(16)) = \{ 0 \}$, whence $M_2(\Gamma_0(16))$ is spanned by Eisenstein series as follows:
\[ M_{2}(\Gamma_0(16)) = \big(\bigoplus_{1 < t | 16} \mathbb C E_2^{1, 1, t} \big) \oplus \mathbb C E_2^{\chi_{-4}, \chi_{-4}, 1}. \]
In terms of the identity
\[ \sigma_{1}^{\chi_{-4}, \chi_{-4}}(n) \, = \, \big( \frac{-4}{n} \big) \sigma(n), \]
we obtain for every $n \in \mathbb N$

\begin{eqnarray*}
r_{1, 0, 0, 0; 2}(n) & = & \big(1 + \big(\frac{-4}{n}\big)\big) \sigma(n) -3\sigma(n/2) + 2\sigma(n/4), \\
r_{1, 1, 0, 0; 2}(n) & = & 4\sigma(n/2) - 12\sigma(n/4) + 8\sigma(n/8), \\
r_{1, 1, 1, 0; 2}(n) & = & \big(1 - \big(\frac{-4}{n}\big)\big) \sigma(n) -3\sigma(n/2) + 2\sigma(n/4), \\
r_{1, 1, 1, 1; 2}(n) & = & 16\sigma(n/4) - 48\sigma(n/8) + 32\sigma(n/16).
\end{eqnarray*}
\end{example}

\begin{example}
By a routine computation as above we can obtain the following results.

(1) The modular form $f_Q^{u; 2}(\tau)$ with $Q(x) = x_1^2 + x_2^2 + x_3^2 + 2x_4^2$ is contained in
\[ M_2(\Gamma_0(32), \chi_2) = \big( \bigoplus_{a = 1, 2, 4} \mathbb C E_2^{1, \chi_2, a} \big) \oplus \big( \bigoplus_{b = 1, 2, 4} \mathbb C E_2^{\chi_2, 1, b} \big) \oplus \mathbb C E_2^{\chi_{-2}, \chi_{-4}, 1} \oplus \mathbb C E_2^{\chi_{-4}, \chi_{-2}, 1}. \]
We thus obtain for every $n \in \mathbb N$
\begin{eqnarray*}
r_Q^{1, 0, 0, 0; 2}(n) & = & \sum_{d|n} \big( \frac{2}{n/d} \big)d - 2 \sum_{d|\frac{n}{2}} \big( \frac{2}{n/2d} \big)d +\big( \frac{-2}{n} \big) \sum_{d|n} \big( \frac{2}{d} \big) d, \\
r_Q^{1, 1, 0, 0; 2}(n) & = & 4 \sum_{d|\frac{n}{2}} \big( \frac{2}{n/2d} \big)d -8 \sum_{d|\frac{n}{4}} \big( \frac{-2}{n/4d} \big) d, \\
r_Q^{1, 1, 1, 0; 2}(n) & = & - \big( 1 + \big( \frac{-2}{n} \big) \big) \sum_{d|n} \big( \frac{2}{d} \big)d +\sum_{d|\frac{n}{2}} \big( \frac{2}{d} \big)d + \big( 1 + \big( \frac{-2}{n} \big) \big) \sum_{d|n} \big( \frac{2}{n/d} \big)d - 2\sum_{d|\frac{n}{2}} \big( \frac{2}{n/2d} \big)d, \\
r_Q^{0, 0, 0, 1; 2}(n) & = & -2 \sum_{d|\frac{n}{2}} \big( \frac{2}{d} \big)d + 2\sum_{d|\frac{n}{4}} \big( \frac{2}{d} \big)d + 4\sum_{d|\frac{n}{2}} \big( \frac{2}{n/2d} \big)d - 8\sum_{d|\frac{n}{4}} \big( \frac{2}{n/4d} \big)d, \\
r_Q^{1, 0, 0, 1; 2}(n) & = & \sum_{d|n} \big( \frac{2}{n/d} \big)d - 2\sum_{d|\frac{n}{2}} \big( \frac{2}{n/2d} \big)d - \big( \frac{-2}{n} \big) \sum_{d|n} \big( \frac{2}{d} \big) d, \\
r_Q^{1, 1, 0, 1; 2}(n) & = & 8\sum_{d|\frac{n}{4}} \big( \frac{2}{n/4d} \big)d, \\
r_Q^{1, 1, 1, 1; 2}(n) & = & -\big( 1 - \big( \frac{-2}{n} \big) \big) \sum_{d|n} \big( \frac{2}{d} \big)d +\sum_{d|\frac{n}{2}} \big( \frac{2}{d} \big)d + \big( 1 - \big( \frac{-2}{n} \big) \big) \sum_{d|n} \big( \frac{2}{n/d} \big)d - 2\sum_{d|\frac{n}{2}} \big( \frac{2}{n/2d} \big)d.
\end{eqnarray*}

(2) Let $Q(x) = x_1^2 + x_1x_2 + x_2^2 + x_3^2 + x_3x_4 + x_4^2$. We see
\[ r_Q^{u_1, u_2, u_3, u_4; 2}(n) = r_Q^{u_2, u_1, u_3, u_4; 2}(n) = r_Q^{u_3, u_4, u_1, u_2; 2}(n) = r_Q^{u_1 + u_2, -u_2, u_3, u_4; 2}(n) \]
and obtain for every $n \in \mathbb N$
\begin{eqnarray*}
r_Q^{1, 0, 0, 0; 2}(n) & = & 2\sigma(n) -6\sigma(n/2) - 6\sigma(n/3) + 4\sigma(n/4)+18\sigma(n/6)-12\sigma(n/12), \\
r_Q^{1, 0, 1, 0; 2}(n) & = & 4\sigma(n/2) - 4\sigma(n/4) - 12\sigma(n/6) + 12 \sigma(n/12).
\end{eqnarray*}

(3) If $Q(x) = x_1^2 + x_2^2 + x_3^2 + x_4^2 + x_1x_2 + x_1x_3 + x_1x_4$, then a computation shows that $r_Q^{u; 2}(n)$ with $u = (u_1, \ldots, u_4) \in \{0, 1\}^4$ and $n \in \mathbb N$ is given as
\[ \left\{ \begin{array}{ll}
8\sigma(n/2) - 24\sigma(n/4) + 16\sigma(n/8) & \mbox{if $u_1 = 0$ and exactly one of $u_2, u_3, u_4$ is $0$} \\
24 \sigma(n/4) - 48 \sigma(n/8) & \mbox{if $u_1 = u_2 = u_3 = u_4 = 0$} \\
2\sigma(n) - 6\sigma(n/2) + 4\sigma(n/4) & \mbox{otherwise.}
\end{array}\right. \]
\end{example}

\begin{example}
We deal with $Q(x) = x_1^2 + x_2^2 + x_3^2 + x_4^2$. The modular form $f_Q^{u; 3}(\tau)$ is then contained in $M_{2}(\Gamma_0(36))$ by Theorems \ref{Theorem - main} and \ref{Theorem - decomposition}. Unlike the previous examples, it turns out that this space contains a cusp form $\eta(6\tau)^4$, where $\eta(\tau)$ denotes the Dedekind eta function defined by
\[ \eta(\tau) \, = \, q^{1/24} \prod_{n=1}^\infty (1 - q^n). \]
In fact, we have
\[ M_{2}(\Gamma_0(36)) = \big( \bigoplus_{1< a |36} \mathbb C E_2^{1, 1, a} \big) \oplus \big( \bigoplus_{b = 1, 2, 4} \mathbb C E_2^{\chi_{-3}, \chi_{-3}, b} \big) \oplus \mathbb C \eta(6\tau)^4. \]
If we denote by $c(n)$ the $n$th coefficient of the cusp form $\eta(6\tau)^4$, i.e.
\[ \eta(6\tau)^4 = \sum_{n=1}^\infty c(n) q^n = q - 4q^7 + 2q^{13} + 8q^{19} - 5q^{25}  - 4q^{31} - 10 q^{37} + \cdots, \]
then we deduce for every $n \in \mathbb N$
\begin{eqnarray*}
r_Q^{1, 0, 0, 0; 3}(n) & = & \frac{1}{6}\big(1 + \big(\frac{-3}{n}\big)\big) \sigma(n) - \frac{2}{3} \sigma(n/3) - \frac{2}{3} \big(1 + \big(\frac{-3}{n/4}\big)\big) \sigma(n/4) \\
& & + \, \frac{1}{2} \sigma(n/9) + \frac{8}{3} \sigma(n/12) - 2 \sigma(n/36) + \frac{2}{3} c(n), \\
r_Q^{1, 1, 0, 0; 3}(n) & = & \frac{1}{6}\big(1 - \big(\frac{-3}{n}\big)\big) \sigma(n) - \frac{2}{3} \sigma(n/3) - \frac{2}{3} \big(1 - \big(\frac{-3}{n/4}\big)\big) \sigma(n/4) \\
& & + \, \frac{1}{2} \sigma(n/9) + \frac{8}{3} \sigma(n/12) - 2 \sigma(n/36), \\
r_Q^{1, 1, 1, 0; 3}(n) & = & \sigma(n/3) - \sigma(n/9) - 4 \sigma(n/12) + 4 \sigma(n/36), \\
r_Q^{1, 1, 1, 1; 3}(n) & = & \frac{1}{6}\big(1 + \big(\frac{-3}{n}\big)\big) \sigma(n) - \frac{2}{3} \sigma(n/3) - \frac{2}{3} \big(1 + \big(\frac{-3}{n/4}\big)\big) \sigma(n/4) \\
& & + \, \frac{1}{2} \sigma(n/9) + \frac{8}{3} \sigma(n/12) - 2 \sigma(n/36) - \frac{1}{3} c(n).
\end{eqnarray*}
\end{example}

\begin{example}
For $Q(x) = x_1^2 + x_2^2 + 2x_3^2 + 2x_4^2$, we see that $f_Q^{u; 2}(\tau)$ is contained in the space $M_2(\Gamma_0(32))$. Its basis is given as
\[ M_2(\Gamma_0(32)) = \big( \bigoplus_{1<a|32} \mathbb C E_2^{1, 1, a} \big) \oplus \big( \bigoplus_{b=1, 2} \mathbb C E_2^{\chi_{-4}, \chi_{-4}, b} \big) \oplus \mathbb C \eta(4\tau)^2 \eta(8\tau)^2. \]

Let $c(n)$ denote the $n$th coefficient of the Fourier expansion of the cusp form $\eta(4\tau)^2 \eta(8\tau)^2$, namely,
\[ \eta(4\tau)^2 \eta(8\tau)^2 = \sum_{n=1}^\infty c(n) q^n = q - 2q^5 - 3q^9 + 6q^{13} + 2q^{17} - q^{25} + \cdots. \]
We then find for every $n \in \mathbb N$
\begin{eqnarray*}
r_Q^{1, 0, 0, 0; 2}(n) & = & \frac{1}{2} \big( 1 + \big( \frac{-4}{n} \big) \big) \sigma(n) - \frac{3}{2} \sigma(n/2) + \sigma(n/4) + c(n), \\
r_Q^{1, 1, 0, 0; 2}(n) & = & 2 \big( 1 + \big( \frac{-4}{n/2} \big) \big) \sigma(n/2) - 6\sigma(n/4) + 4\sigma(n/8), \\
r_Q^{0, 0, 1, 0; 2}(n) & = & 2\sigma(n/2) - 6\sigma(n/4) + 4\sigma(n/8), \\
r_Q^{1, 0, 1, 0; 2}(n) & = & \frac{1}{2} \big( 1 - \big( \frac{-4}{n} \big) \big) \sigma(n) - \frac{3}{2} \sigma(n/2) + \sigma(n/4), \\
r_Q^{1, 1, 1, 0; 2}(n) & = & 8\sigma(n/4) - 24\sigma(n/8) + 16\sigma(n/16), \\
r_Q^{0, 0, 1, 1; 2}(n) & = & 4\sigma(n/4) + 4\sigma(n/8) - 40\sigma(n/16) + 32 \sigma(n/32), \\
r_Q^{1, 0, 1, 1; 2}(n) & = & \frac{1}{2} \big( 1 + \big( \frac{-4}{n} \big) \big) \sigma(n) - \frac{3}{2} \sigma(n/2) + \sigma(n/4) - c(n), \\
r_Q^{1, 1, 1, 1; 2}(n) & = & 2 \big( 1 - \big( \frac{-4}{n/2} \big) \big) \sigma(n/2) - 6\sigma(n/4) + 4\sigma(n/8).
\end{eqnarray*}
\end{example}

\begin{example}
Consider the quadratic form $Q(x) = x_1^2 + x_1x_2 + x_2^2 + x_3^2 + x_3x_4 + x_4^2$. By definition we see that
\[ r_Q^{u_1, u_2, u_3, u_4; 3}(n) = r_Q^{\pm u_2, \pm u_1, u_3, u_4; 3}(n) = r_Q^{u_3, u_4, u_1, u_2; 3}(n) = r_Q^{u_1+u_2, -u_2, u_3, u_4; 3}(n) \] and we infer that $f_Q^{u; 3}(\tau)$ is contained in $M_{2}(\Gamma_0(27))$ whose basis is given as
\[ M_{2}(\Gamma_0(27)) = \big( \bigoplus_{a = 3, 9, 27} \mathbb C E_2^{1, 1, a} \big) \oplus \big( \bigoplus_{b = 1, 3} \mathbb C E_2^{\chi_{-3}, \chi_{-3}, b} \big) \oplus \mathbb C \eta(3\tau)^2 \eta(9\tau)^2. \]
If we set
\[ \eta(3\tau)^2 \eta(9\tau)^2 = \sum_{n = 1}^\infty c(n) q^n = q - 2q^4 - q^7 + 5q^{13} + 4q^{16} - 7 q^{19} - 5q^{25} + \cdots, \]
then we find for any $n \in \mathbb N$
\begin{eqnarray*}
r_Q^{1, 0, 0, 0; 3}(n) & = & \frac{1}{6} \big( 1 + \big( \frac{-3}{n} \big) \big) \sigma(n) - \frac{2}{3} \sigma(n/3) + \frac{1}{2} \sigma(n/9) + \frac{2}{3} c(n), \\
r_Q^{1, 1, 0, 0; 3}(n) & = & \frac{3}{2} \big( 1 + \big( \frac{-3}{n/3} \big) \big) \sigma(n/3) - 6 \sigma(n/9) + \frac{9}{2} \sigma(n/27), \\
r_Q^{1, 1, 1, 0; 3}(n) & = & \frac{1}{6} \big( 1 + \big( \frac{-3}{n} \big) \big) \sigma(n) - \frac{2}{3} \sigma(n/3) + \frac{1}{2} \sigma(n/9) - \frac{1}{3} c(n), \\
r_Q^{1, 1, 1, 1; 3}(n) & = & \frac{3}{2} \big( 1 - \big( \frac{-3}{n/3} \big) \big) \sigma(n/3) - 6 \sigma(n/9) + \frac{9}{2} \sigma(n/27), \\
r_Q^{1, 0, 1, 0; 3}(n) & = & \frac{1}{6} \big( 1 - \big( \frac{-3}{n} \big) \big) \sigma(n) - \frac{2}{3} \sigma(n/3) + \frac{1}{2} \sigma(n/9).
\end{eqnarray*}
\end{example}

\begin{example}
Let $Q(x) = x_1^2 + x_1x_2 + 2x_2^2 + x_3^2 + x_3x_4 + 2x_4^2$. We easily infer from the definition of $r_Q^{u; 2}(n)$ that
\[ r_Q^{u_1, u_2, u_3, u_4; 2}(n) = r_Q^{u_3, u_4, u_1, u_2; 2}(n) = r_Q^{u_1+u_2, -u_2, u_3, u_4; 2}(n). \]
It follows that $f_Q^{u; 2}(\tau)$ is contained in the space $M_{2}(\Gamma_0(28))$ whose basis is given as
\[ M_{2}(\Gamma_0(28)) = \big( \bigoplus_{1<t|28} \mathbb C E_2^{1, 1, t} \big) \oplus \mathbb C \eta(\tau)\eta(2\tau)\eta(7\tau) \eta(14\tau) \oplus \mathbb C \eta(2\tau)\eta(4\tau)\eta(14\tau) \eta(28\tau). \]
If we set
\[ \eta(\tau)\eta(2\tau)\eta(7\tau) \eta(14\tau) = \sum_{n = 1}^\infty c(n) q^n = q - q^2 - 2q^3 + q^4 + 2q^6 + q^7 - q^8 + \cdots, \]
then one has for every $n \in \mathbb N$
\begin{eqnarray*}
r_Q^{1, 0, 0, 0; 2}(n) & = & \frac{2}{3}\sigma(n) - 2\sigma(n/2) + \frac{4}{3}\sigma(n/4) - \frac{14}{3}\sigma(n/7) + 14\sigma(n/14) \\
& & - \, \frac{28}{3}\sigma(n/28) + \frac{4}{3} c(n) + \frac{4}{3} c(n/2), \\
r_Q^{1, 0, 1, 0; 2}(n) & = & \frac{4}{3}\sigma(n/2) - \frac{4}{3}\sigma(n/4) - \frac{28}{3}\sigma(n/14) + \frac{28}{3} \sigma(n/28) + \frac{8}{3} c(n/2), \\
r_Q^{1, 0, 0, 1; 2}(n) & = & \frac{2}{3}\sigma(n) - 2\sigma(n/2) + \frac{4}{3}\sigma(n/4) - \frac{14}{3}\sigma(n/7) + 14\sigma(n/14) \\
& & - \, \frac{28}{3}\sigma(n/28) - \frac{2}{3} c(n) - \frac{2}{3} c(n/2), \\
r_Q^{0, 1, 0, 0; 2}(n) & = & \frac{4}{3}\sigma(n/2) - \frac{4}{3}\sigma(n/4) - \frac{28}{3}\sigma(n/14) + \frac{28}{3} \sigma(n/28) + \frac{2}{3} c(n/2), \\
r_Q^{0, 1, 0, 1; 2}(n) & = & \frac{4}{3}\sigma(n/2) - \frac{4}{3}\sigma(n/4) - \frac{28}{3}\sigma(n/14) + \frac{28}{3} \sigma(n/28) - \frac{4}{3} c(n/2).
\end{eqnarray*}
\end{example}

\bibliographystyle{amsplain}

\end{document}